\documentclass[12pt]{article}

\usepackage{geometry}
\usepackage{graphicx}
\usepackage{amsmath}
\usepackage{amsfonts}
\usepackage{amsthm}
\usepackage{todonotes}

\usepackage[colorlinks=true,citecolor=black,linkcolor=black,urlcolor=blue]{hyperref}

\theoremstyle{plain}

\newtheorem{defn}[equation]{Definition}

\newtheorem{thm}[equation]{Theorem}

\newtheorem{cor}[equation]{Corollary}

\newtheorem{example}[equation]{Example}

\makeatletter\@addtoreset{equation}{section}\makeatother

\begin{document}

\title{A Note on the Structure of\\ Roller Coaster Permutations}

\author{William L. Adamczak\\ 
\small Department of Mathematics, Siena College \\
\small Loudonville, New York 12211 USA\\
\small\tt wadamczak@siena.edu
}
\date{}

\maketitle

\begin{abstract}
We consider the structure of roller coaster permutations as introduced by Ahmed \& Snevily\cite{ahsn}. A roller coaster permutation is described as a permutation that maximizes the total switches from ascending to descending or visa versa for the permutation and all of its subpermutations simultaneously. This paper looks at the alternating structure of these permutations and then we introduce a notion of a condition stronger than alternating for a permutation that we shall refer to as recursively alternating. We also examine the behavior of what entries can show up in even, odd, and end positions within the permutations.

\end{abstract}

\section{Introduction}

The idea of roller coaster permutations first shows up in a work of Ahmed \& Snevily \cite{ahsn} where they are described as a permutations that maximize the total switches from ascending to descending or visa versa for the permutation and all of its subpermutations simultaneously. More basically, this gets the greatest number of ups and downs for the permutation  and all possible subpermutations.

 These permutations have strong relations to pattern avoiding permutations, alternating permutations alone have connections to permutation avoidance as seen in Mansour \cite{mans} in the context of avoiding 132. The connection with forbidden subsequences and partitions of permutations is seen in Stankova \cite{stank} where particular forbidden subsequences end up being roller coaster permutations. \\
  
Throughout this paper we will use one-line notation for permutations. Let $[n] = \{1,2, \dots, n\}$. A permutation $\pi$ will be viewed as a sequence $(\pi_1,\pi_2,\ldots,\pi_n)$ where the commas may be omitted for small $n$. Let $S_n$ denote the set of all permutations of $[n]$.\\

\begin{defn}
We will use the following definitions given by Ahmed \& Snevily \cite{ahsn}:\\

\hspace{1cm} $i(\pi) = $ the number of increasing sequences of contiguous numbers in $\pi$,\\

\hspace{1cm} $d(\pi) = $ the number of decreasing sequences of contiguous numbers in $\pi$,\\

\hspace{1cm} $id(\pi) = i(\pi) + d(\pi) $ ,\\

\hspace{1cm} $X(\pi) = \{\tau : \tau \text{ is a subsequence of } \pi \text{ such that } \left|\tau\right| \geq 3\} $ , \\

\hspace{1cm} $\displaystyle{t(\pi) = \sum\limits_{\tau \in X(\pi)} id(\tau)}$ .\\

\noindent
Here we refer to contiguous numbers as consisting of at least two numbers.\\
\end{defn}

\begin{example}
	Consider the permutation $3412$ in $S_4$ where we have\\
	
	$X(1324) = \{1324, 132, 134, 412\}$\\
	
	$t(1324) = id(1324)+id(132)+id(134)+id(412) = 3 + 2 + 1 + 2 = 8$.
\end{example}
\vspace{1mm}

\begin{defn}
	
Let $RC(n) = \{\pi \in S_n : t(\pi) = \max\limits_{\sigma \in S_n} t(\sigma) \}$.\\

The members of $RC(n)$ are then referred to as a \emph{Roller Coaster Permutations}.\\

\end{defn}

\begin{example} For $n = 3,4,5,6$ we have: \\
	
	$RC(3) = \{132,213,231,312 \}$
	
	$RC(4) = \{2143,2413,3142,3412\}$
	
	$RC(5) = \{24153, 25143, 31524, 32514, 34152, 35142, 41523, 42513 \}$
	
	$RC(6) = \{326154, 351624, 426153, 451623\}$\\
	
	\textit{For a more extensive listing see \cite{ahsn}.}
	
\end{example}

\section{Structure of Roller Coaster Permutations}

Ahmed \& Snevily \cite{ahsn} conjectured that all roller coaster permutations are either \emph{alternating} or \emph{reverse-alternating} where they defined these as $\pi_1 < \pi_2 > \pi_3 < \ldots$ and $\pi_1 > \pi_2 < \pi_3 > \ldots$ respectively. Where there is no confusion we will simply refer to such a permutation as \emph{alternating} regardless of whether it is alternating or reverse alternating. \\

\begin{thm}
	If $\pi \in RC(n)$ then $\pi$ is alternating.\\
\end{thm}

\begin{proof}

Consider the rightmost positions $i-1$, $i$, $i+1$ such that $\pi$ is alternating from there to the right. Let us assume these are two consecutive increases, the case for two consecutive decreases will follow analogously. Consider $\tilde{\pi} = \pi\cdot(i,i+1)$, ie exchanging positions $i$ and $i+1$. The we note that subsequences $\sigma$ of $\pi$ involving none of these have no change in $id(\sigma)$.\\

Note that for the cases where only one of the positions $i$ or $i+1$ is involved in $\sigma$ since these are merely swapped and as such still show up as subpermutations. We may therefore consider only subpermutations containing these two consecutive positions. For any such subpermutation $\sigma$ let positions $i-1,i,i+1,i+2$ be such that $\sigma_{i} = \tilde{\pi}_{i}$ and $\sigma_{i+1} = \tilde{\pi}_{i+1}$. \\
 
Now note that the worst case is to have a subpermutation $\sigma$ that $\sigma_{i-1} > \sigma_{i} > \sigma_{i+1} > \sigma_{i+1}$ as this would decrease the ascent/descent count by two. Note that $\sigma_{i-1}$ cannot be $\pi_{i-1}$ since $\pi_{i-1} < \pi_{i}$ and $\pi_{i-1} < \pi_{i+1}$. Now note that for every subsequence $\sigma$ not involving $\pi_{i}$ there is a correspondence subsequence involving $\pi_{i}$. Those without may decrease the ascent/descent count by at most two, however any subsequence involving $\pi_{i}$ will certainly increase the ascent/descent count by two, so at worst these subsequences will not increase the total count, but this cannot decrease.\\

Now observe that we have not yet accounted for the subsequence involving only positions $i$ through $i+2$ which sees an increase in the ascent/descent count by one, thus the total count must increase meaning that the total ascent/descent count is greater for $\tilde{\pi}$ than for $\pi$, ie $t(\tilde{\pi}) > t(\pi)$, which contradicts the assumption that $\pi$ is in $RC(n)$.

\end{proof}

\vspace{1mm}

The next property of note regarding roller coaster permutations is that the value in the first and final positions of the permutation must differ by exactly one. This will be used in building up to the proof of the odd-sums conjecture of Ahmed \& Snevily \cite{ahsn} for roller coaster permutations as well as in proving a deeper self-similar structure to be introduced.

\begin{thm}
	For $\pi \in RC(n)$ $ \left| \pi_n -\pi_1 \right| =1$.
\end{thm}

\begin{proof}
	Suppose that $\pi$ does not satisfy this condition. Then there exists an $i$, $ 1<i<n$, such that either $\pi_1 < \pi_i < \pi_n$ or $\pi_1 > \pi_i > \pi_n$. Without loss of generality we may assume $\pi_1 < \pi_i < \pi_n$ since the reverse is merely an operation on the reverse of the permuation which is of course also in $RC(n)$. \\
	
	Pick $i$ the be the position corresponding to the largest value of $\pi_i$ such that the inequality $\pi_1 < \pi_i < \pi_n$ holds. Then consider the permutation $\pi' = \pi \cdot (i,n)$. We claim that $t(\pi')> t(\pi)$ which would contradict the assumption that $\pi \in RC(n)$.\\
	
	Note that for $1<j \neq i <n$ we have that either $\pi_j < \pi_i, \pi_n$ or $\pi_j > \pi_i, \pi_n$. Thus for any subsequence other than $\pi_1,\pi_i,\pi_n$ the total number of ascents and descents remains unchanged since the inequalities relative to the positions remain unchanged otherwise. Note however that the subsequence $\pi'_1,\pi'_i,\pi'_n$ increases the count of ascents and descents by one, thereby establishing the claim.
\end{proof}

 We next would like to observe how the values in odd positions and the values in even positions show up in roller coaster permutations. \\

\begin{thm}
	For $\pi \in RC(n)$ and $ \pi $ an alternating (not reverse-alternating) permutation $\pi_i > \pi_1,\pi_n$ for $i$ even and $\pi_i < \pi_1,\pi_n$ for $i$ odd.
	For $\pi \in RC(n)$ and $ \pi $ a reverse-alternating permutation $\pi_i < \pi_1,\pi_n$ for $i$ even and $\pi_i > \pi_1,\pi_n$ for $i$ odd.
\end{thm}

\begin{proof}
We will show this in the context of an alternating permutation in $RC(n)$ as the proof for reverse-alternating is analogous.\\
Begin by finding the leftmost indices $j,j+1$ where both $\pi_{j}$ and $\pi_{j+1}$ are greater than the maximum of $\pi_1$ and $\pi_n$. Note that $\pi_{j} > \pi_{j+1}$ by the choice of the leftmost pair. Also note that by alternating we have that $\pi_{j+2}$ is also greater than the maximum of $\pi_1$ and $\pi_n$.\\

Let $\sigma = \pi (l,r)$ where the index $l$ is such that $\pi_l = max \{\pi_i \;|\; 1 \leq i <j , \pi_1,\pi_n \leq \pi_i <\pi_j \}$. Then observe that $t(\sigma) > t(\pi)$ via an argument analogous to that used in proving that these permutations are alternating since the index $l$ selected for make the exchange is chosen precisely to avoid interfering with subsequences not involving the exchanged positions. This contradicts that $\pi \in RC(n)$.

\end{proof}

To simplify the idea, this is basically stating that all the odd positioned entires have values that are all greater than the start and ending value of the permutation or all less than, and the reverse for the even positioned entries. Based on this result we quickly get major insights into the structure of these permutations. Breaking these into the cases where $n$ is odd/even and the permutation is alternating versus reverse-alternating  will add additional structural insight which follows easily from the results to this point.

\begin{cor}
	For $\pi \in RC(n)$ with $n$ odd and $ \pi $ an alternating permutation we have that:
	 \begin{itemize}
	 	\item $\{\pi_1, \pi_n\} = \{\frac{n-1}{2},\frac{n+1}{2} \}$
	 	\item $\{\pi_j \; | \; 3\leq j \leq n-2, \; j \mbox{ odd} \} = \{1, \ldots, \frac{n-3}{2} \}$
	 	\item $\{\pi_j \; | \; 2\leq j \leq n-1, \; j \mbox{ even} \} = \{\frac{n+3}{2}, \ldots, n \}$
	 \end{itemize}
	 
	 For $\pi \in RC(n)$ with $n$ odd and $ \pi $ a reverse-alternating permutation we have that:
	 \begin{itemize}
	 	\item $\{\pi_1, \pi_n\} = \{\frac{n+1}{2},\frac{n+3}{2} \}$
	 	\item $\{\pi_j \; | \; 2\leq j \leq n-1, \; j \mbox{ odd} \} = \{\frac{n+5}{2}, \ldots, n \}$
	 	\item $\{\pi_j \; | \; 3\leq j \leq n-2, \; j \mbox{ even} \} = \{1, \ldots, \frac{n-1}{2} \}$
	 \end{itemize}
	 
	 For $\pi \in RC(n)$ with $n$ even and $ \pi $ an alternating permutation we have that:
	 \begin{itemize}
	 	\item $\{\pi_1, \pi_n\} = \{\frac{n}{2},\frac{n}{2}+1 \}$
	 	\item $\{\pi_j \; | \; 3\leq j \leq n-2, \; j \mbox{ odd} \} = \{1,\ldots, \frac{n}{2}-1 \}$
	 	\item $\{\pi_j \; | \; 2\leq j \leq n-1, \; j \mbox{ even} \} = \{\frac{n}{2}+2, \ldots, n \}$
	 \end{itemize}
	 
	 For $\pi \in RC(n)$ with $n$ even and $ \pi $ a reverse-alternating permutation we have that:
	 \begin{itemize}
	 	\item $\{\pi_1, \pi_n\} = \{\frac{n}{2},\frac{n+2}{2} \}$
	 	\item $\{\pi_j \; | \; 2\leq j \leq n-1, \; j \mbox{ odd} \} = \{\frac{n}{2}+2, \ldots, n \}$
	 	\item $\{\pi_j \; | \; 3\leq j \leq n-2, \; j \mbox{ even} \} = \{1,\ldots, \frac{n}{2}-1 \}$
	 \end{itemize}
	 
\end{cor}

\begin{proof}
	Merely observe that the previous result together with the fact that the start and end of the permutation differ by one gives that all values greater than the start and end are either in even or odd positions depending on whether this is alternating or reverse-alternating. 
\end{proof}
\vspace{2mm}

Note here that for $n$ odd a permutation and its reverse will be both alternating or both reverse-alternating. When $n$ is even this is no longer the case, which starts to lend some insight into how different the order of $RC(n)$ is in the odd versus even cases.\\

While we have seen that roller coaster permutations are in fact alternating we will prove a stronger condition that we shall refer to as recursively alternating. We will use the notation $\pi[i,j]$ to refer to the restriction of the permutation $\pi$ to positions $i$ through $j$.\\

\begin{defn}
First define $\pi_{odd}$ as the restriction to the odd indexed positions of $\pi[2,n-1]$ and $\pi_{even}$ as the restriction to the even indexed positions of $\pi[2,n-1]$.\\

A permutation $\pi$ is said to be \emph{recursively alternating} if:

\begin{itemize}
	\item 
	$\pi$ is alternating or reverse alternating.
	\item 
	$\pi_{odd}$ and $\pi_{even}$ are alternating or reverse alternating.
	\item
	Recursively iterating this process yields alternating or reverse alternating, ie for $\pi_{odd, odd}$ and $\pi_{odd,even}$, etc.

\end{itemize}

\end{defn}

\begin{thm}
	If $\pi \in RC(n)$ then $\pi$ is recursively alternating.
\end{thm}

\begin{proof}
Assume $\pi \in RC(n)$ and is therefore alternating or reverse-alternating. Consider $\pi_{odd}$ and assume $\pi_{odd}$ is not alternating or reverse-alternating as the $\pi_{even}$ case is analogous. 
Apply the same argument used to show that roller coaster permutations were alternating noting that this modification to $\pi_{odd}$ does not effect the direction of inequalities relative to other positions of $\pi$ since all values in odd positions are either all greater than or all less than the values of all even positions.
Lastly iterate the argument as required.

\end{proof}

\bibliographystyle{plain}

\end{document}